\documentclass{amsart}
\usepackage{graphicx}
\usepackage{amssymb}
\usepackage{epstopdf}
\usepackage{amsmath}
\usepackage{amsthm}
\usepackage[numbers]{natbib}

\newtheorem{theorem}{Theorem}

\theoremstyle{definition}
\newtheorem{definition}[theorem]{Definition}
\newtheorem{prop}[theorem]{Proposition}
\newtheorem{lem}[theorem]{Lemma}

\newtheorem{xca}[theorem]{Exercise}
\newtheorem{cor}[theorem]{Corollary}
\newtheorem{prob}[theorem]{Problem}

\def\vp{\varphi}
\def\al{\alpha}
\def\om{\omega}

\def\fa{\forall}
\def\ex{\exists}
\def\iff{\Longleftrightarrow}

\def\ar{arithmetic}
\def\ct{countable}
\def\au{automorphism}

\def\fo{first-order}

\def\be{\begin{xca}}
\def\ee{\end{xca}}

\def\rs{recursively saturated}
\def\lpa{{\mathcal L}_{\sf PA}}

\DeclareMathOperator{\dcl}{dcl}

\begin{document}
\title{Neutrally Expandable Models of Arithmetic}
\author{Athar Abdul-Quader}
\author{Roman Kossak}

\date{December 18, 2017}                                           

\begin{abstract}
A subset of a model of {\sf PA} is called neutral if it does not change the $\dcl$ relation. A model with undefinable neutral classes is called neutrally expandable. We study the existence and non-existence of neutral sets in various models of {\sf PA}. We show that cofinal extensions of prime models are neutrally expandable, and $\om_1$-like neutrally expandable models exist, while no recursively saturated model is neutrally expandable. We also show that neutrality is not a first-order property. In the last section, we study a local version of neutral expandability.
\end{abstract}


\maketitle

\section{Introduction}
All models in this note are models of {\sf PA} and their expansions. We will use $\mathcal M$, $\mathcal N$, $\mathcal K$, etc.~for models of {\sf PA}, and $M$, $N$, $K$, etc.~for their respective domains. We will write ${\mathcal M}\prec_{\rm end}{\mathcal N}$ if ${\mathcal N}$ is an elementary end extension of ${\mathcal M}$, ${\mathcal M}\prec_{\rm cof}{\mathcal N}$ if ${\mathcal N}$ is a cofinal extension of ${\mathcal M}$, and
${\mathcal M}\prec_{\rm cons}{\mathcal N}$ if ${\mathcal N}$ is an elementary conservative extension of ${\mathcal M}$.

A subset $X$ of a model ${\mathcal M}$ is a \emph{class} if for each $a\in M$, $\{x \in X : {\mathcal M}\models x<a\}$ is definable in ${\mathcal M}$. A subset $X$ of $M$ is inductive if $({\mathcal M},X)\models {\sf PA}^*$, i.e. the induction schema holds in $({\mathcal M},X)$ for all formulas of  the language of {\sf PA} with a unary predicate symbol interpreted as $X$. All inductive sets are classes. 

Simpson proved that every \ct\ model ${\mathcal M}$ has an inductive undefinable subset $X$ such that $({\mathcal M},X)$ is pointwise definable \cite{sim}.  Enayat showed that there are non-prime models ${\mathcal M}$, such that for every undefinable class $X$, $({\mathcal M},X)$ is pointwise definable \cite{ali}. The first author called such models Enayat, and studied them in detail in \cite{ath}.

In search for a general notion of a generic subset of a model of {\sf PA}, Alf Dolich suggested the definition below. Since the term ``generic" has been already used in \ar\ context, we will use another name.

\begin{definition} A subset $X$ of $M$ is  \emph{neutral} if for all $a$ in $M$, the definable closure of $a$ in ${\mathcal M}$, $\dcl^{\mathcal M}(a)$, and the definable closure of $a$ in $({\mathcal M},X)$, $\dcl^{({\mathcal M},X)}(a)$,  are the same. We will call a model \emph{neutrally expandable} if it has an undefinable neutral class.
\end{definition}

For trivial reasons, prime models are neutrally expandable. Every subset of a prime model is neutral. In every model, 0-definable sets are neutral. It is less trivial that the standard cut is neutral in every model; this result follows from a generalization of a theorem of Kanovei \cite{kan}, \cite[Theorem 8.4.7]{ks}. Constructing neutral undefinable classes in models that are not prime is a harder task. 

\section{Neutral Inductive Sets}\label{s1}
If $X$ is a neutral inductive subset of $M$, then for each ${\mathcal K}\prec {\mathcal M}$, $K$ is closed under the Skolem terms of $\lpa \cup \{ X \}$, i.e.~the language of {\sf PA} with the unary relation symbol interpreted in ${\mathcal M}$ as $X$. Hence; by the Tarski-Vaught test, we have the following characterization of inductive neutral sets.
 
\begin{prop}\label{ind} If $X$ is an inductive subset of $M$, then $X$ is neutral iff for every ${\mathcal K}\prec {\mathcal M}$, $({\mathcal K},X\cap K)\prec ({\mathcal M},X)$. 
\end{prop}

Proposition \ref{ind} has an immediate corollary.

\begin{cor} If $X$ is a neutral  inductive subset of $M$, and ${\mathcal K}\prec {\mathcal M}$, then $X\cap K$ is an inductive neutral subset of ${\mathcal K}$.
\end{cor}

\begin{prop}\label{prop1} If ${\mathcal M}$ is prime and ${\mathcal M}\prec_{\rm cof} {\mathcal N}$, then ${\mathcal N}$ is neutrally expandable.
\end{prop}
\begin{proof} Let $X$ be an undefinable  inductive subset  of $M$. Such sets always exist; for example we can take $X$ to be a generic subset of $M$, see \cite{sim} or \cite[Chapter 6]{ks}.  

By the Kotlarski-Schmerl lemma \cite[Theorem 1.3.7]{ks}, there is a unique $Y\subseteq N$ such that $({\mathcal M},X)\prec ({\mathcal N},Y)$, and it follows  that $Y$ is an undefinable inductive subset of $N$. 

Suppose that for $a$ and $b$ in $N$, $a\in \dcl^{({\mathcal N},Y)}(b)$. Then, there is a Skolem term $t$ such that $a=t(b,Y)$.
Fix $c\in M$ such that $a, b<c$. There is $d\in M$ such that
\[({\mathcal M},X)\models \fa x,y<c [x=t(y,X)\iff x=(d)_y].\]
Since $({\mathcal M},X)\prec ({\mathcal N},Y)$, it follows that $a=(d)_b$, proving that $a\in \dcl(b)$.
\end{proof}
Let us note that in the proof of Proposition \ref{prop1}, the witness $Y$ to neutral expandability of ${\mathcal N}$ is a generic.

If ${\mathcal M}$ is \ct, and $X\subseteq M$ is coded in a \ct\ elementary end extension ${\mathcal N}$,   then $X$ can be extended to a generic subset of $N$. If $X$ is generic, and  $Y$ is a generic subset of $N$ extending $X$, then $({\mathcal M},X)\prec ({\mathcal N},Y)$ (see \cite[Corollary 6.2.8]{ks}). 

An extension ${\mathcal N}$ of a model ${\mathcal M}$ is {\it superminimal} if for each $a\in N\setminus M$, $N=\dcl(a)$. It follows directly from definitions that if $X$ is a neutral  subset of $M$, ${\mathcal N}$ is a superminimal elementary  extension of ${\mathcal M}$, and $({\mathcal M},X)\prec ({\mathcal N},Y)$, then $Y$ is a neutral subset of $N$. 

Theorem 5 of \cite{sch} implies that  any inductive subset of a \ct\ model can be coded in an elementary superminimal end extension. Hence, we have the following lemma.

\begin{lem}\label{lemB} If  $X$ is a generic neutral  subset of a \ct\ model ${\mathcal M}$, then there  is $({\mathcal N},Y)$ such that   $({\mathcal M},X)\prec_{\rm end} ({\mathcal N},Y)$, $Y$ is neutral in ${\mathcal N}$, and ${\mathcal N}$ is a superminimal extension of ${\mathcal M}$.
\end{lem}

\begin{theorem}\label{olike} Every completion of {\sf PA} has $\om_1$-like neutrally expandable models.
\end{theorem}
\begin{proof} Start with any generic subset $X_0$ of a prime model ${\mathcal M}_0$ and iterate Lemma \ref{lemB} along $\om_1$. That is, let $\{({\mathcal M}_\alpha, X_\alpha)\}_{\al\in \om_1}$ be a continuous elementary chain  such that ${\mathcal M}_{\al+1}$ is a superminimal elementary end extension of ${\mathcal M}_\al$, and $X_\al$ is undefinable and neutral in ${\mathcal M}_\al$ that is coded in ${\mathcal M}_{\al+1}$.  Let ${\mathcal N} = \bigcup\limits_{\alpha < \om_1} {\mathcal M}_\alpha$ and $X = \bigcup\limits_{\alpha < \om_1} X_\alpha$. Since each ${\mathcal M}_{\alpha + 1}$ is a superminimal extension of ${\mathcal M}_\al$, every elementary submodel of ${\mathcal N}$ is ${\mathcal M}_\alpha$ for some $\alpha < \om_1$. Further, for each $\alpha < \om_1$, $X_\alpha = X \cap M_\alpha$, so $({\mathcal M}_\alpha, X \cap M_\alpha) \prec ({\mathcal N}, X)$ for each $\alpha < \om_1$; hence $X$ is undefinable and inductive. Since, for every ${\mathcal K} \prec {\mathcal N}$, we  have  $({\mathcal K}, X \cap K) \prec ({\mathcal N}, X)$,   by Proposition \ref{ind} it follows that  $X$ is neutral.
\end{proof}

Lemma \ref{lemB} shows that if $\mathcal{M}$ is countable and $X \subseteq M$ is a neutral generic subset of $M$, then there is a countable $(\mathcal{N}, Y)$ such that $Y$ is a neutral generic subset of $N$ and $(\mathcal{M}, X) \prec_\text{end} (\mathcal{N}, Y)$.  We do not know whether every neutrally expandable model has a neutrally expandable elementary end extension. A positive answer to the following question would allow us to generalize Theorem \ref{olike} to $\kappa$-like modes.

\begin{prob}  Let   $X\subseteq M$ be inductive and neutral. Are there $\mathcal N$ and $Y\subseteq N$, such that $({\mathcal M},X)\prec_{\rm end} ({\mathcal N},Y)$ and $Y$ is neutral in $\mathcal N$? 
\end{prob}

\section{Neutral Classes}

 If $X\subseteq N$ is inductive and neutral, and  ${\mathcal M}$ is such that ${\mathcal M}\prec_{\rm cons}{\mathcal N}$, then $X\cap M$ is definable in ${\mathcal M}$ and it follows from Proposition \ref{ind} that $X$ is definable in ${\mathcal N}$. We will show that the same result holds under the assumption that $X$ is a class.

\begin{lem}\label{class} If ${\mathcal M}\prec {\mathcal N}$ and $X$ is a neutral class of ${\mathcal N}$, then $X\cap M$ is a class of  ${\mathcal M}$.
\end{lem}
\begin{proof} Let $X$ be a neutral class of ${\mathcal N}$.  Let \[m(x)=\min\{y: \fa z<x (z\in y \iff z\in X)\}.\]
For each $c\in N$, $m(c)$ is well-defined in ${\mathcal N}$ because $X$ is a class, and for $c\in M$, $m(c)\in M$ because $X$ is a neutral class. 
\end{proof}

\begin{theorem}\label{thm1} Let $X$ be a class of a model ${\mathcal N}$. If there is ${\mathcal M}\prec {\mathcal N}$ such that ${\mathcal M}$ is a conservative extension of its prime submodel, then $X$ is a neutral class iff $X$ is 0-definable in ${\mathcal N}$.
\end{theorem}
\begin{proof}  Since all 0-definable sets are neutral, we only need to prove one  direction of the theorem.  Let ${\mathcal M}_0$ be the prime elementary submodel of ${\mathcal N}$, and suppose that $X$ is a neutral class of ${\mathcal N}$. By Lemma \ref{class}, $X\cap M$ is a class of ${\mathcal M}$, and, since ${\mathcal M_0}\prec_{\rm cons} {\mathcal M}$, it follows that $X\cap M_0$ is 0-definable. 
Let $\vp(x)$ be a formula defining $X\cap M_0$ in ${\mathcal M}_0$. If $\vp({\mathcal N})\not=X$,  then
\[e=\min\{y : \vp(y) \iff y \not\in X\}\] is well-defined in ${\mathcal N}$  because $X$ is a class, and $e>M_0$, which contradicts neutrality of $X$. 
\end{proof}

For the  proofs of the next result and the results in Section \ref{s2} we will use Gaifman's minimal types. See Chapter 3 of \cite{ks}, in particular Theorem 3.2.10, for all necessary background.

Since every \rs\ model realizes minimal types and elementary end extensions generated by an element realizing a minimal type are conservative,  we have the following corollary.

\begin{cor}\label{cor 2}
If ${\mathcal N}$ has a \rs\ elementary submodel, and $X$ is a class of ${\mathcal N}$, then  $X$ is a neutral class iff  $X$ is 0-definable. 
\end{cor}

By a result of Kotlarski, Krajewski, and Lachlan, each \ct\ \rs\ model of {\sf PA} has a full satisfaction class \cite{kkl}. Satisfaction classes are never definable, and they may not be classes; however, Smith showed \cite[Theorem 2.10]{smi} that if $S$ is a full satisfaction class for a model ${\mathcal M}$, then ${\mathcal M}$ has an undefinable class that is definable in $({\mathcal M},S)$. This result and Theorem \ref{thm1} imply that  no full satisfaction class is neutral.

The following terminology was introduced by Smory\'nski \cite{smo}. A model $\mathcal M$ is {\it short} if for some $a\in M$, $\dcl(a)$ is cofinal in $M$. A model that is not short is {\it tall}. A model $\mathcal M$ is {\it short recursively saturated} if it realizes all finitely realizable recursive types that contain a formula $x<b$ for some $b\in \mathcal M$. It is easy to see that every  tall short \rs\ model is \rs. 

Let  ${\mathcal N}$ be \rs.  For $a\in N$ let $I({\mathcal N};a)$ be the elementary submodel of ${\mathcal N}$ whose domain is $\{x: \ex b\in \dcl(a)\ {\mathcal N}\models x<b\}$. All models of the form $I({\mathcal N};a)$ are short \rs.
Every short, short \rs\ model is of the form $I({\mathcal N};a)$ for a \rs\ $\mathcal N$. This was proved for \ct\ models in \cite[Theorem C]{smo}, and in full generality in \cite[Theorem 5.1]{kos}.

For every parameter free complete type $p(x)$ that is realized  in ${\mathcal N}$ by some $a>I({\mathcal N};0)$, the set of its realizations of $p(x)$ in ${\mathcal N}$ is coinitial with  $I({\mathcal N};0)$. Hence,  Corollary \ref{cor 2} applies to all models $I({\mathcal N};a)$, for $a>I({\mathcal N};0)$.  Models of the form $I({\mathcal N};0)$ are exceptional. We know from Proposition \ref{prop1} that $I({\mathcal N};0)$ is neutrally expandable. In fact, we can show more: if $I({\mathcal N};0)$ is nonstandard, it is an example of a model that has a neutral class that is not inductive.

\begin{prop}\label{prop2} Let ${\mathcal N}$ be short \rs\ and let ${\mathcal N}_0$ be the prime elementary submodel of $\mathcal N$. Then, every subset of $\mathcal N_0$ is neutral in $\mathcal N$. 
\end{prop}
\begin{proof}  Let $X$ be a subset of $N_0$ and let  $a,b\in  N\setminus N_0$ be such that $a\notin \dcl(b)$.  If ${\mathcal N}$ is un\ct, then instead of ${\mathcal N}$ we can take  a \ct\ short \rs\ model ${\mathcal N}'$ such that $a,b\in N'$ and $({\mathcal N}',X)\prec ({\mathcal N},X)$. So let us assume that ${\mathcal N}$ is \ct.  

If $\mathcal N$ is tall, and hence \rs, then there is an \au\ $f$ of ${\mathcal N}$ that fixes $b$ and moves $a$. Since $f$  fixes $N_0$ pointwise and fixes $I({\mathcal N};0)$ setwise,  $a\notin \dcl^{({\mathcal N},X)}(b)$. If $\mathcal N$  is short, then it has a \ct\  \rs\ elementary end extension $\mathcal K$, and we can repeat the same argument working in $\mathcal K$.
\end{proof}

In Proposition \ref{prop2}, instead of requiring that $\mathcal N$ is short \rs\  we could assume that for all $a,b\in N$ such that $a\notin \dcl(b)$, there is an \au\ of $\mathcal N$ fixing $b$ and moving $a$. Models with that property do not have to be short \rs.

For an example of a non-prime model that has a neutral class that is not inductive, we can apply Proposition \ref{prop2} to a \ct\ short \rs\ model $\mathcal N$ that is a cofinal extension of ${\mathcal N}_0$ and any unbounded $X\subseteq N_0$ of order type $\om$.

\begin{prob}\label{prob}  Are there neutral non-inductive classes other than those given by Proposition \ref{prop2}? \end{prob}

Let ${\mathcal K}$ and ${\mathcal N}$ be \rs\ models such that ${\mathcal K}\prec_{\rm cof}{\mathcal N}$, and such that the standard cut is strong in ${\mathcal K}$ but not in ${\mathcal N}$. Let ${\mathcal K}_0$ be the prime elementary submodel of $\mathcal K$, and let $X$ be an unbounded subset of $K_0$ of order type $\om$. Then, $X$ is a neutral class of $I({\mathcal K};0)$ and $I({\mathcal N};0)$, and since $\om$ is definable (by the same definition) in both models, it follows that $(I({\mathcal N};0),X)$ is not an elementary extension of $(I({\mathcal K};0),X)$. This shows that Proposition \ref{ind} does not hold for classes, but still we have a weaker version of it that is a corollary of Lemma \ref{class} and its proof.

\begin{prop}\label{prop3} If $X$ is a neutral unbounded class of ${\mathcal N}$, then for each ${\mathcal K}\prec {\mathcal N}$, $X \cap K$ is an unbounded class of ${\mathcal K}$.
\end{prop}

Proposition \ref{prop3} may be prove to be useful, but it may also turn out to hold trivially if the answer to Problem \ref{prob} is negative.

Let us finish this part with an obvious question.

\begin{prob} Are there neutrally expandable models other than those shown in this section?
\end{prob}

\subsection{Neutrality is not first-order}

If $X$ is a neutral subset of a model $\mathcal M$, and $({\mathcal M}, X)\prec ({\mathcal N}, Y)$, then $Y$ may not be neutral. An easy example is given by the standard model ${\mathbb N}$, an undefinable set of natural numbers $X$, and $({\mathcal N},Y)$ that is a \rs\ elementary extension of $({\mathbb N},X)$. $Y$ is an inductive undefinable subset of $N$, but it is not neutral, because ${\mathcal N}$, being \rs, does not have neutral classes.  

Another set of examples is given by the tall neutrally expandable models constructed in the proof of Theorem \ref{olike}. It is not difficult to see that every tall model has a cofinal elementary extension that is \rs. Hence, each model ${\mathcal M}_\lambda$, for limit $\lambda<\om_1$, has a cofinal \ct\ \rs\ extension ${\mathcal N}_\lambda$. By the Kotlarski-Schmerl lemma, there is a unique $Y_\lambda$, such that $({\mathcal M}_\lambda, X_\lambda)\prec ({\mathcal N}_\lambda, Y_\lambda)$. While $X_\lambda$ is neutral in ${\mathcal M}_\lambda$,  $Y_\lambda$ is not neutral in ${\mathcal N}_\lambda$.

Dolich has asked whether there is a theory $T$ such that for every \rs\ model ${\mathcal M}$ (of any \fo\ theory), and any set $X$, $X$ is neutral iff $({\mathcal M},X)\models T$. We provide a (partial) negative answer here.

\begin{cor}There is no  theory $T$ extending $\mathsf{PA}$ such that for any \rs\ ${\mathcal M}$ and any set $X$, $X$ is neutral iff $({\mathcal M}, X) \models T$.
\end{cor}

\begin{proof} Suppose $T$ is such a theory. Corollary \ref{cor 2} implies that if ${\mathcal M}$ is a countable recursively saturated model of {\sf PA}, then for any subset $X$ of ${\mathcal M}$,
\begin{displaymath} 
\textup{$X$ is 0-definable iff }({\mathcal M}, X) \models T + \textup{$X$ is a class}.
\end{displaymath}
Let ${\mathcal M}$ be a \rs\ model such that $T$ is coded by a set in the standard system of ${\mathcal M}$. Let $S$ be the theory 
$T$  extended by ``$X$  is a class"  and  the sentences of the form $\ex x \lnot [x\in X \iff \vp(x)]$, for all formulas  $\vp(x)$ of the language of ${\sf PA}$.
Since ${\mathcal M}$ is resplendent and $S$ is in the standard system of ${\mathcal M}$, one can find an expansion $({\mathcal M}, Y) \models S$ (see \cite[Theorem  1.9.3]{ks}). Then,  $Y$ is a neutral class and is not  0-definable, which is a contradiction.
\end{proof}

\section{A-Neutral Sets}\label{s2}

\begin{definition} Let $A \subseteq M$. A subset $X$ of ${\mathcal M}$ is called \emph{$A$-neutral} if, for all $a, b \in A$, $a \in \dcl^{\mathcal M}(b)$ iff $a \in \dcl^{({\mathcal M}, X)}(b)$.
\end{definition}

The two theorems below are a strengthening of the results we previously obtained for finite $A$, and for coded
$\om$-sequences. The proofs are due to Jim Schmerl, and they are included here with his kind permission. 

\begin{theorem} Let  ${\mathcal M}$ be countable and recursively saturated and let $A$ be a bounded subset of $M$. Then ${\mathcal M}$ has an inductive, undefinable $A$-neutral subset $X$.
\end{theorem}

\begin{proof}
Without loss of generality, we can assume that $A$ is an elementary cut of $\mathcal M$. Let $p(x)$ be a minimal type realized in ${\mathcal M}$, and let $C \subseteq M \setminus A$ be a cofinal set of realizations of $p(x)$. Let $N = \dcl(A \cup C)$. Since minimal types are definable, we have ${\mathcal A} \prec_\text{cons} {\mathcal N} \prec_\text{cof} {\mathcal M}$. First we notice that if 
\begin{displaymath}
Y = \{ x : {\mathcal N} \models \vp(x, a, \bar{c}) \},
\end{displaymath}
where $a \in A$ and $\bar{c}$ is a tuple of elements of $C$, then $Y \cap A$ is definable using only $a$ as a parameter.

Let $G \subseteq \dcl(C)$ be generic. By the Kotlarski-Schmerl lemma, there is $X \subseteq M$ such that $(\dcl(C), G) \prec ({\mathcal M}, X)$. We notice that $G$ is cofinal in $X$ and, since $G$ is generic, $X$ is also generic. Let $a, b \in A$ be such that $a \in \dcl^{({\mathcal M}, X)}(b)$. That is, there is a formula $\vp(x, y)$ in the language $\lpa \cup \{ X \}$ such that 
\begin{displaymath}
({\mathcal M}, X) \models \vp(a, b) \wedge \exists !x \vp(x, b).
\end{displaymath}

By the forcing lemma for arithmetic \cite[Lemma 6.2.6]{ks}, there is a relation $\Vdash$ definable in $\lpa$ such that for each  formula $\sigma(x, y)$ of $\lpa \cup \{ X \}$,  and for all parameters $m, n \in M$, there is $p \in X$ such that
\begin{displaymath}
({\mathcal M}, X) \models \sigma(m, n) \textup{ iff }{\mathcal M} \models p \Vdash \sigma(m, n).
\end{displaymath}
Let $p \in X$ be such that ${\mathcal M} \models p \Vdash [\vp(a, b) \wedge \exists! x \vp(x, b)]$. Since $G$ is cofinal in $X$, there is $q \in G$ extending $p$, so ${\mathcal M} \models q \Vdash [\vp(a, b) \wedge \exists! x \vp(x, b)]$. Let
\begin{displaymath}
Y = \{ \langle x, y \rangle : {\mathcal N} \models q \Vdash [\vp(x, y) \wedge \exists! z \vp(z, y)] \}
\end{displaymath}
Notice that $Y$ is definable using only the parameter $q \in G \subseteq C$, so $Y \cap A$ is definable in $A$ without parameters. Since $a$ is unique such that $\langle a, b \rangle \in Y \cap A$, $a \in \dcl(b)$.
\end{proof}

We can also find cofinal subsets $A$ of countable, recursively saturated models for which there exist $A$-neutral sets.

\begin{prop}If ${\mathcal M}$ is countable and recursively saturated, then there are  a cofinal  $A\subseteq M$ and an undefinable inductive  $X\subseteq M$ that is $A$-neutral.
\end{prop}

\begin{proof}
Let $X$ be an inductive, undefinable subset of $M$ such that $({\mathcal M}, X)$ is recursively saturated. Such $X$ exists by chronic resplendence of ${\mathcal M}$. Let $p(x)$ be a minimal type in the language $\lpa \cup \{ X \}$ realized in $({\mathcal M}, X)$. Let $A$ be the set of realizations of $p(x)$ in ${\mathcal M}$. Then  $A$ is cofinal in $M$. We will show that $X$ is $A$-neutral.
To see this, let $a, b \in A$ and suppose $a \in \dcl^{({\mathcal M}, X)}(b)$. By Ehrenfeucht's Lemma \cite{ehr} (see also \cite[Theorem 1.7.2]{ks}),  if $a \neq b$, then $\mathrm{tp}^{({\mathcal M},X)}(a) \neq \mathrm{tp}^{({\mathcal M},X)}(b)$. Since $a$ and $b$ realize the same type, this means $a = b$.
\end{proof}

In the previous result, we found a particular cofinal subset $A$ of a recursively saturated model ${\mathcal M}$ for which there exists an $A$-neutral inductive subset of $M$. Our last question asks if there are cofinal elementary submodels ${\mathcal A}$ of ${\mathcal M}$ for which we have $A$-neutral inductive subsets.

\begin{prob}If ${\mathcal M}$ is recursively saturated, is there ${\mathcal A} \prec_\text{cof} {\mathcal M}$ such that ${\mathcal M}$ has an undefinable $A$-neutral inductive subset?
\end{prob}

\bibliographystyle{plain}
\bibliography{references}

\section*{Acknowledgements} Jim Schmerl's  comments on a preliminary version of this paper allowed us to improve some of the results and the overall presentation. Thank you Jim.

\end{document}